\documentclass[10pt,reqno]{amsart}

\usepackage[pass]{geometry}
\newlength\DX
\paperwidth=\dimexpr\paperwidth-\DX\relax
\hoffset=\dimexpr\hoffset-.5\DX\relax
\newlength\DY
\paperheight=\dimexpr\paperheight-\DY\relax
\voffset=\dimexpr\voffset-.5\DY-.5\footskip\relax

\usepackage{mathtools}

\usepackage{lmodern}
\usepackage{mathabx}

\usepackage[T1]{fontenc}

\usepackage{amsmath,amsfonts,amsbsy,amsgen,amscd,mathrsfs,amssymb,amsthm}
\usepackage{enumerate}
\usepackage{bm}

\usepackage{stmaryrd}
\SetSymbolFont{stmry}{bold}{U}{stmry}{m}{n} 

\usepackage[usenames,dvipsnames]{xcolor}
\usepackage[colorlinks=true,citecolor=blue,linkcolor=blue]{hyperref}

\usepackage{tikz}

\newtheorem{thm}{Theorem}[section]
\newtheorem{lem}[thm]{Lemma}
\newtheorem{conj}[thm]{Conjecture}
\newtheorem{prop}[thm]{Proposition}

\theoremstyle{definition}

\theoremstyle{remark}

\newtheorem{example}[thm]{Example}
\newtheorem{rem}[thm]{Remark}
\newtheorem*{rem*}{Remark}

\numberwithin{equation}{section} 
\numberwithin{figure}{section}
\numberwithin{table}{section}

\makeatletter
\renewcommand\subsubsection{\@startsection{subsubsection}{3}%
  \z@{.5\linespacing\@plus.7\linespacing}{-.5em}%
  {\normalfont\bfseries}}
\makeatother

\newcommand{\Vol}{\mathrm{Vol}}
\newcommand{\V}{\mathsf{V}}

\newcommand{\MM}{\mathrm{M}}
\newcommand{\hh}{\mathrm{h}}

\begin{document}

\title[On a conjecture of Fedotov]{Shephard's inequalities, Hodge-Riemann 
relations, and a conjecture of Fedotov}

\author{Ramon van Handel}
\address{Fine Hall 207, Princeton University, Princeton, NJ 
08544, USA}

\begin{abstract}
A well-known family of determinantal inequalities for mixed volumes of 
convex bodies were derived by Shephard from the Alexandrov-Fenchel 
inequality. The classic monograph \emph{Geometric Inequalities} by Burago 
and Zalgaller states a conjecture on the validity of higher-order 
analogues of Shephard's inequalities, which is attributed to Fedotov. In 
this note we disprove Fedotov's conjecture by showing that it contradicts 
the Hodge-Riemann relations for simple convex polytopes. Along the way, we 
make some expository remarks on the linear algebraic and geometric aspects 
of these inequalities.
\end{abstract}

\subjclass[2010]{52A39; 
                 52A40} 

\keywords{Mixed volumes; Alexandrov-Fenchel inequality; Shephard's 
inequalities; Hodge-Riemann relations for convex polytopes}

\maketitle

\thispagestyle{empty}

\section{Introduction}
\label{sec:intro}

\subsection{}

Let $K_1,\ldots,K_m$ be convex bodies in $\mathbb{R}^n$ and 
$\lambda_1,\ldots,\lambda_m>0$. One of the most basic facts of convex 
geometry, due to H.\ Minkowski, is that the volume of convex bodies
is a homogeneous polynomial in the sense that
$$
	\Vol(\lambda_1K_1+\cdots+\lambda_m K_m)
	= \sum_{i_1,\ldots,i_n=1}^m \V(K_{i_1},\ldots,K_{i_n})\,
	\lambda_{i_1}\cdots\lambda_{i_n}.
$$
The coefficients $\V(K_1,\ldots,K_n)$, called mixed volumes, define 
a large family of natural geometric parameters of convex bodies, and play 
a central role in convex geometry \cite{BZ88,Sch14}. Mixed 
volumes are always nonnegative, are symmetric in their arguments, and are 
additive and homogeneous in each argument.

The fundamental inequality in the theory of mixed volumes is the
following.

\begin{thm}[Alexandrov-Fenchel] 
\label{thm:af}
For convex bodies
$K,L,C_1,\ldots,C_{n-2}$ in $\mathbb{R}^n$
$$
	\V(K,L,C_1,\ldots,C_{n-2})^2 \ge
	\V(K,K,C_1,\ldots,C_{n-2})\,
	\V(L,L,C_1,\ldots,C_{n-2}).
$$
\end{thm}

Numerous inequalities in convex geometry may be derived from the
Alexandrov-Fenchel inequality, 
cf.\ \cite[\S 20]{BZ88} and \cite[\S 7.4]{Sch14}. The starting point for 
this note is a well-known family of determinantal inequalities, due to 
Shephard \cite{She60}, that extend the Alexandrov-Fenchel inequality to 
more than $n$ bodies.

\begin{thm}[Shephard]
\label{thm:shep}
Given convex bodies 
$K_1,\ldots,K_m,C_1,\ldots,C_{n-2}$ in $\mathbb{R}^n$, define the
$m\times m$ symmetric matrix $\MM$ by setting
$$
	\MM_{ij} := \V(K_i,K_j,C_1,\ldots,C_{n-2}).
$$
Then
$$
	(-1)^m\det\MM\le 0.
$$
\end{thm}

The special case $m=2$ of Theorem \ref{thm:shep} is just a reformulation 
of the Alexandrov-Fenchel inequality, and Shephard's inequalities may thus 
be viewed as a considerable generalization of the Alexandrov-Fenchel 
inequality. However, as is shown by Shephard (and as we will explain later 
in this note), the general inequalities may in fact be deduced from the 
$m=2$ case by a simple linear algebraic argument. In the case $m=3$, this 
result dates back already to Minkowski \cite[p.\ 478]{Min03}.

\subsection{}

The classic monograph \emph{Geometric Inequalities} by Burago and 
Zalgaller states a conjecture on the validity of a higher-order 
generalization of Theorem \ref{thm:shep}, which is attributed to Fedotov 
\cite[\S 20.6]{BZ88}. Let us recall the statement of this conjecture.
In the sequel, we will frequently employ the notation
$$
	\V(K_1[m_1],K_2[m_2],\ldots,K_r[m_r]) :=
	\V(
	\underbrace{K_1,\ldots,K_1}_{m_1},
	\underbrace{K_2,\ldots,K_2}_{m_2},
	\ldots,
	\underbrace{K_r,\ldots,K_r}_{m_r})
$$
when convex bodies are repeated multiple times
in the arguments of a mixed volume.

\begin{conj}[Fedotov]
\label{conj:fed}
Let $k\le n/2$, and let
$K_1,\ldots,K_m,C_1,\ldots,C_{n-2k}$ be convex bodies in $\mathbb{R}^n$.
Define the
$m\times m$ symmetric matrix $\MM$ by setting
$$
	\MM_{ij} := \V(
	K_i[k],K_j[k],C_1,\ldots,C_{n-2k}).
$$
Then
$$
	(-1)^m\det\MM\le 0.
$$
\end{conj}

If true, this conjecture would entail a considerable generalization of 
Shephard's inequalities. The conjecture is rather appealing, as it is 
easily verified to be true in two extreme cases that have a different 
flavor.

\begin{lem}
\label{lem:fedeasy}
Conjecture \ref{conj:fed} is valid in the following two cases:
\begin{enumerate}[a.]
\item When $k=1$ and $m$ is arbitrary.
\item When $m=2$ and $k$ is arbitrary.
\end{enumerate}
\end{lem}

\begin{proof}
Case $a$ is nothing other than Theorem \ref{thm:shep}. To prove $b$, 
it suffices to note that iterating the Alexandrov-Fenchel inequality
yields \cite[(7.63)]{Sch14}
\begin{multline*}
	\V(K_1[k],K_2[l],C_1,\ldots,C_{n-k-l})^{k+l}
	\ge \\
	\V(K_1[k+l],C_1,\ldots,C_{n-k-l})^k\,
	\V(K_2[k+l],C_1,\ldots,C_{n-k-l})^l
\end{multline*}
for any $k,l\ge 1$, $k+l\le n$. The case $k=l$ is readily seen to be 
equivalent to $b$.
\end{proof}

The main purpose of this note is to explain that Conjecture \ref{conj:fed} 
fails when one goes beyond the special cases of Lemma 
\ref{lem:fedeasy}. More precisely, we will prove:

\begin{thm}
\label{thm:main}
For every $k>1$, 
Conjecture \ref{conj:fed} is false for some $m>2$.
\end{thm}

\subsection{}

In order to explain how we will disprove Conjecture \ref{conj:fed}, it is 
useful to first briefly recall some of its history. 

Despite the fundamental nature of the Alexandrov-Fenchel inequality, no 
really elementary proof of it is known. Alexandrov gave two different (but 
closely related) proofs in the 1930s: a combinatorial proof using strongly 
isomorphic polytopes \cite{Ale37}, and an analytic proof using elliptic 
operators \cite{Ale38}. Further remarks on its history and on more 
modern proofs may be found in \cite{Sch14,SvH18}.

In the 1970s, unexpected connections were discovered between the theory of 
mixed volumes and algebraic geometry. In particular, a remarkable identity 
due to Bernstein and Kushnirenko \cite[Theorem 27.1.2]{BZ88} shows that 
the number of solutions $z\in(\mathbb{C}\backslash\{0\})^n$ of a generic 
system of polynomial equations $p_1(z)=0,\ldots,p_n(z)=0$ with given 
monomials coincides with the mixed volume of an associated family of 
lattice polytopes in $\mathbb{R}^n$ (i.e., polytopes with vertices in 
$\mathbb{Z}^n$).

Motivated by these developments, Fedotov \cite{Fed79} proposed a simple 
proof of the Alexandrov-Fenchel inequality using only basic properties of 
polynomials. Fedotov further notes that his method even yields the more 
general Conjecture \ref{conj:fed}, which is stated in \cite{Fed79} as a 
theorem. These results were included in the Russian edition of the 
monograph of Burago and Zalgaller. Unfortunately, Fedotov's elementary 
approach turns out to contain a serious flaw, which renders his method of 
proof invalid. A correct algebraic proof of the Alexandrov-Fenchel 
inequality was given by Teissier and Khovanskii using nontrivial 
machinery, namely a reduction to the Hodge index theorem of algebraic 
geometry. The latter proof is included in the English translation of 
Burago-Zalgaller \cite[\S 27]{BZ88}, but does not settle the validity of 
Fedotov's higher-order analogue of Shephard's inequalities \cite[\S 
20.6]{BZ88}.

On the other hand, the algebraic connection yields other higher-order 
inequalities. The Alexandrov-Fenchel inequality is analogous to a 
Hodge-Riemann relation of degree $1$ in the cohomology ring of a smooth 
projective variety \cite{DN06,Huh18}. Hodge-Riemann relations of higher 
degree give rise to new inequalities in convex geometry. Such inequalities 
were first stated by McMullen \cite{McM93} for strongly isomorphic simple 
polytopes as a byproduct of his work on the $g$-conjecture. Their 
geometric significance was greatly clarified by Timorin \cite{Tim99}, 
whose formulation is readily interpreted in terms of explicit inequalities 
for mixed volumes. Very recently, some special cases were extended also to 
smooth convex bodies in \cite{Kot20,Ale20,KW22}.

The proof of Theorem \ref{thm:main} may now be explained as follows. Using 
the properties of hyperbolic quadratic forms, we will first reformulate 
Conjecture \ref{conj:fed} as a higher-order Alexandrov-Fenchel inequality. 
In this equivalent formulation, it will be evident that this inequality 
contradicts the Hodge-Riemann relation of degree $2$. Thus the results of 
McMullen and Timorin imply that Conjecture \ref{conj:fed} is false. Beside 
disproving the conjecture, a more expository aim of this note is to draw 
attention to some basic linear algebraic and geometric aspects of the 
above inequalities (none of which are really new here) in the context of 
classical convexity.

\begin{rem*}
It should be noted that Fedotov's conjecture as stated in 
\cite[\S 20.6]{BZ88} is somewhat more general than Conjecture 
\ref{conj:fed}: the matrix $\MM$ considered there is
$$
	\MM_{ij} := \V(
	K_i[k],K_j[l],C_1,\ldots,C_{n-k-l})
$$
for any $k,l\ge 1$ such that $k+l\le n$. Lemma 
\ref{lem:fedeasy} extends to this setting: the case $k=l=1$ and 
general $m$ reduces to Shephard's inequalities, while the case $m=2$ and 
general $k,l$ is obtained by multiplying the inequality 
\cite[(7.63)]{Sch14} used in the proof of Lemma \ref{lem:fedeasy} by the 
same inequality with the roles of $k,l$ reversed.
When $k\ne l$, however, the matrix $\MM$ is not symmetric, and the 
spectral interpretation of the conjecture becomes unclear. Given that we 
show the conjecture fails for general $m$ already in the symmetric case 
$k=l$, it seems implausible that the nonsymmetric case $k\ne l$ has any 
merit, and we do not consider it further in this note.
\end{rem*}

\subsection{}

The remainder of this note is organized as follows. In section 
\ref{sec:lin}, we recall some basic properties of hyperbolic quadratic 
forms that will be used in the sequel. We also briefly discuss Shephard's 
inequalities and clarify their equality cases. In section~\ref{sec:hodge} 
we formulate the Hodge-Riemann relations for strongly isomorphic simple 
polytopes, due to McMullen and Timorin, entirely in the language of 
classical convexity. Finally, section \ref{sec:main} completes the proof 
of Theorem \ref{thm:main}.

While the proof of Theorem \ref{thm:main} explains clearly \emph{why} 
Fedotov's conjecture must fail, the construction is rather indirect. Once 
the proof has been understood, however, it is not difficult to engineer
an explicit counterexample, which will be done in section 
\ref{sec:explicit}. Beside further illustrating the basic construction, 
this example will show that we may in fact choose $m=3$ in Theorem 
\ref{thm:main}.

We conclude this note by highlighting a puzzling aspect of the 
Hodge-Riemann relations: even though their statement makes sense in 
principle for arbitrary convex bodies, the Hodge-Riemann relations have 
only been proved for special classes of bodies (e.g., strongly isomorphic 
simple polytopes). In section \ref{sec:hrgeneral}, we will illustrate by 
means of a simple example that the Hodge-Riemann relations may fail for 
general convex bodies. This highlights the rather unusual nature of the 
Hodge-Riemann relations as compared to other inequalities in convex 
geometry.

\section{Linear algebra}
\label{sec:lin}

The aim of this section is to explain that the connection between the 
Alexandrov-Fenchel and Shephard inequalities has nothing to do with 
convexity, but is rather a simple linear-algebraic fact. The results of 
this section are known in various forms, see, e.g., \cite[Theorem 
4.4.6]{BR97}, \cite[Lemma 2.9]{SvH18}, or \cite[Lemma 3.1]{SvH19}, but we 
provide simple self-contained proofs for the variants needed here.

\subsection{Hyperbolic matrices}

We begin by giving a spectral interpretation of the Alexandrov-Fenchel 
inequality. In the sequel, a matrix $\MM$ will be called positive if 
$\MM_{ij}>0$ for all $i,j$. For $y\in\mathbb{R}^m$, we write $y\ge 0$ 
($y>0$) if $y_i\ge 0$ ($y_i>0$) for all $i$. The linear span of all 
eigenvectors of a symmetric matrix $\MM$ with positive eigenvalues will be 
called the positive eigenspace of $\MM$.

\begin{lem}
\label{lem:hyp}
Let $\MM$ be a symmetric positive matrix.
The following are equivalent:
\begin{enumerate}[1.]
\item The positive eigenspace of $\MM$ is one-dimensional.
\item $\langle x,\MM y\rangle^2 \ge \langle x,\MM x\rangle\,
\langle y,\MM y\rangle$ for all $x\ge 0$ and $y\ge0$.
\item $\langle x,\MM y\rangle=0$ implies $\langle x,\MM x\rangle\le 0$
for all $x$ and $y\ge 0$, $y\ne 0$.
\end{enumerate}
\end{lem}

\begin{proof} 
As $\MM$ is a positive matrix, the Perron-Frobenius theorem implies that 
it has at least one eigenvector $v>0$ with positive eigenvalue.

\medskip

\emph{3}$\Rightarrow$\emph{1}: Let $x\perp v$ be any other eigenvector of 
$\MM$. Then $\langle x,\MM v\rangle = 0$, so \emph{3} implies $\langle 
x,\MM x\rangle\le 0$. Thus the eigenvalue associated to $x$ must be 
nonpositive.

\medskip

\emph{1}$\Rightarrow$\emph{2}: 
It follows from \emph{1} that 
$\MM$ is negative semidefinite on $v^\perp$. Fix $x,y\ge 0$; we may assume 
$y\ne 0$ (else the inequality is trivial), so that $\langle 
y,v\rangle>0$ and $\langle y,\MM y\rangle>0$. If we define $z=x-ay$ with 
$a=\langle x,v\rangle/\langle y,v\rangle$, then $z\in v^\perp$, so
$$
	0\ge \langle z,\MM z\rangle
	= \langle x,\MM x\rangle 
	-2a\langle x,\MM y\rangle + a^2\langle y,\MM y\rangle
	\ge
	\langle x,\MM x\rangle-\frac{\langle x,\MM y\rangle^2}{
	\langle y,\MM y\rangle}.
$$

\medskip

\emph{2}$\Rightarrow$\emph{3}:
We first show that \emph{2} remains valid for any $x$ (not just $x\ge 0$).
Suppose first that $y>0$. Then $x+by\ge 0$ when $b$ 
is chosen sufficiently large, so \emph{2} implies
$$
	\langle x+by,\MM y\rangle^2 \ge
	\langle x+by,\MM (x+by)\rangle \,\langle y,\MM y\rangle.
$$
Expanding both sides of this inequality shows that all terms involving $b$ 
cancel, so
$
	\langle x,\MM y\rangle^2 \ge \langle x,\MM x\rangle\,
	\langle y,\MM y\rangle
$
for any $x$ and $y>0$. This conclusion 
remains valid for any $y\ge 0$ by applying the above argument with 
$y\leftarrow y+\varepsilon v$ and letting $\varepsilon\to 0$. Now
\emph{3} follows immediately once we note 
that $y\ge 0$, $y\ne 0$ implies $\langle y,\MM y\rangle>0$.
\end{proof}

In the sequel, a symmetric (but not necessarily positive) matrix that has 
a one-dimensional positive eigenspace will be called \emph{hyperbolic}.

\subsection{Shephard's inequalities}

An $m\times m$ hyperbolic matrix $\MM$ has $1$ positive and $m-1$ 
nonpositive eigenvalues. It is therefore immediately obvious that such a 
matrix satisfies $(-1)^m\det\MM \le 0$ (as the determinant is the product 
of the eigenvalues). Shephard's inequalities follow directly from this 
observation.

\begin{proof}[Proof of Theorem \ref{thm:shep}]
We may assume without loss of generality that all the convex bodies have 
nonempty interior, so that $\MM$ is a positive matrix (otherwise we may 
replace $K_i\leftarrow K_i+\varepsilon B$, $C_i\leftarrow C_i+\varepsilon 
B$ for any body $B$ with nonempty interior, and take $\varepsilon\to 0$ in 
the final inequality.) Condition \emph{2} of Lemma \ref{lem:hyp} is 
immediate from the Alexandrov-Fenchel inequality (Theorem \ref{thm:af} with
$K=\sum_i x_iK_i$ and $L=\sum_i y_i K_i$). Thus $\MM$ is hyperbolic
by Lemma \ref{lem:hyp}, which implies $(-1)^m\det\MM \le 0$.
\end{proof}

While this is only tangentially related to the rest of this note, let us 
take the opportunity to clarify the cases of equality in Shephard's 
inequalities.

\begin{prop}
\label{prop:shepeq}
In the setting and notations of Theorem \ref{thm:shep}, we have 
$\det\MM=0$ 
if and only if there are linearly independent vectors $x,y>0$ such that 
$K=\sum_i x_iK_i$, $L=\sum_iy_iK_i$ yield equality in the 
Alexandrov-Fenchel inequality of Theorem \ref{thm:af}.
\end{prop}

\begin{proof}
We must show $\det\MM=0$ if and only if $\langle x,\MM 
y\rangle^2=\langle x,\MM x\rangle\,\langle y,\MM y\rangle$ for some 
linearly independent $x,y>0$.  We may assume $\MM\ne 0$ (else
the result is trivial).

Suppose first that $\det\MM=0$. Then there exists $z\in\ker\MM$, $z\ne 0$. 
Choose any $y>0$ that is linearly independent of $z$. Evidently $\langle 
z,\MM y\rangle^2=\langle z,\MM z\rangle\,\langle y,\MM y\rangle$. But as 
this identity is invariant under the replacement $z\leftarrow z+by$ (as in 
the proof of \emph{2}$\Rightarrow$\emph{3} of Lemma \ref{lem:hyp}), we may 
choose $x=z+by>0$ for $b$ sufficiently large.

Now suppose $\langle x,\MM y\rangle^2=\langle x,\MM 
x\rangle\,\langle y,\MM y\rangle$ for linearly independent $x,y>0$.
Then
$$
	q(v) := \langle x+v,\MM y\rangle^2 -
	\langle x+v,\MM (x+v)\rangle\,\langle y,\MM y\rangle
$$
satisfies $q(0)=0$, and $q(v)\ge 0$ for all $v$ in a neighborhood of $0$
by the Alexandrov-Fenchel inequality. Thus
$\nabla q(0)=0$, which yields $z=\langle y,\MM y\rangle x-
\langle x,\MM y\rangle y \in \ker\MM$. Moreover, $z\ne 0$ as
$x,y$ are linearly independent. Thus $\det\MM=0$.
\end{proof}

Proposition \ref{prop:shepeq} reduces the equality cases of Shephard's 
inequalities to those of the Alexandrov-Fenchel inequality. The 
characterization of the latter is a long-standing open problem \cite[\S 
7.6]{Sch14}, which was recently settled in several important cases in 
\cite{SvH20,SvH19}. This problem remains open in full generality.

\subsection{A Sylvester criterion}

While any hyperbolic $m\times m$ matrix $\MM$ trivially satisfies 
$(-1)^m\det\MM \le 0$, the converse implication clearly does not hold: the 
sign of the determinant does not determine the number of positive 
eigenvalues. However, the implication can be reversed if the 
determinant condition holds for all principal submatrices of $\MM$. This 
hyperbolic analogue of the classical Sylvester criterion may be proved in 
essentially the same manner.\footnote{%
	The author learned the elementary approach used here
	from lecture notes of M.\ Hlad\'ik.
}

In the following, we denote for any $m\times m$ symmetric matrix $\MM$ and 
subset $I\subseteq[m]$ by $\MM_I:=(\MM_{ij})_{i,j\in I}$ the associated 
principal submatrix.

\begin{lem}
\label{lem:syl}
For a symmetric positive $m\times m$ matrix,
the following are equivalent:
\begin{enumerate}[1.]
\item The positive eigenspace of $\MM$ is one-dimensional.
\item $(-1)^{|I|}\det\MM_I \le 0$ for all $I\subseteq [m]$.
\end{enumerate}
\end{lem}

\begin{proof}
To prove \emph{1}$\Rightarrow$\emph{2}, note first that condition \emph{2} 
of Lemma \ref{lem:hyp} is inherited by all its principal submatrices
$\MM_I$ (as one may restrict to $x,y$ supported on $I$). The conclusion 
therefore follows immediately from Lemma \ref{lem:hyp}.

To prove \emph{2}$\Rightarrow$\emph{1}, we argue by induction on $m$. For 
$m=2$, it suffices to note that as $\MM$ has at least one positive 
eigenvalue by the Perron-Frobenius theorem, $\det\MM \le 0$ implies that 
its other eigenvalue must be nonpositive.

Now let $m>2$ and assume the result has been proved in dimensions up to 
$m-1$. Then \emph{2} implies that $\MM_I$ is hyperbolic for all 
$I\subsetneq[m]$. By the Perron-Frobenius theorem, $\MM$ has an 
eigenvector $v$ with positive eigenvalue. Now suppose \emph{1} fails, that 
is, $\MM$ is not hyperbolic. Then there must be another eigenvector 
$w\perp v$ with positive eigenvalue. As $(-1)^m\det\MM\le 0$, there must 
then be a third eigenvector $u\perp\{v,w\}$ with nonnegative eigenvalue. 
Choose any $i\in[m]$ such that $u_i\ne 0$ and let $I=[m]\backslash\{i\}$. 
Choose $a,b\in\mathbb{R}$ so that $x:=v-au$ and $y:=w-bu$ satisfy 
$x_i=y_i=0$. By construction, $x,y$ are linearly independent and $\langle 
z,\MM z\rangle>0$ for all $z\in\mathrm{span}\{x,y\}$, $z\ne 0$. As $x,y$ 
are supported on $I$, this implies $\MM_I$ has a positive eigenspace 
of dimension at least two, contradicting the induction hypothesis. 
\end{proof}

It follows immediately from Lemma \ref{lem:syl} that Conjecture 
\ref{conj:fed} is equivalent to the statement that the matrix $\MM$ is 
hyperbolic. This observation will form the basis for the proof of Theorem 
\ref{thm:main} in section \ref{sec:main}: we will show that hyperbolicity 
of $\MM$ contradicts the Hodge-Riemann relations for simple convex 
polytopes.

\section{Hodge-Riemann relations}
\label{sec:hodge}

The Hodge-Riemann relations in algebraic geometry give rise to higher 
order analogues of the Alexandrov-Fenchel inequality \cite{McM93,Tim99}. 
While these inequalities are not usually stated in this form in the 
literature, they may be equivalently formulated as explicit inequalities 
between mixed volumes. The aim of this section is to draw attention to 
this elementary formulation of the Hodge-Riemann relations in terms of 
familiar objects from classical convex geometry.

Recall that a convex polytope in $\mathbb{R}^n$ is called \emph{simple} if 
it has nonempty interior and each vertex is contained in exactly $n$ 
facets. In the following, let us fix an arbitrary simple polytope $\Lambda$ 
in $\mathbb{R}^n$, and denote by $\mathcal{P}(\Lambda)$ the collection of 
polytopes that are \emph{strongly isomorphic} to $\Lambda$: that is,
$P\in\mathcal{P}(\Lambda)$ if and only if 
$$
	\dim F(P,u) = \dim F(\Lambda,u)\quad\mbox{for all }u\in S^{n-1},
$$
where $F(P,u)$ denotes the face of $P$ with normal direction $u$. For the 
basic properties of simple and strongly ismorphic polytopes, the reader is 
referred to \cite[\S 2.4]{Sch14}. For the purposes of this note, the only 
significance of these definitions is that they are needed for the validity 
of the following theorem (see section \ref{sec:hrgeneral}).

\begin{thm}[McMullen-Timorin]
\label{thm:hodge}
Fix $n\ge 2$ and a simple polytope $\Lambda\in\mathbb{R}^n$, and
let $m\ge 1$, $k\le n/2$, 
$K_1,\ldots,K_m,L,C_1,\ldots,C_{n-2k}\in\mathcal{P}(\Lambda)$, 
and $x\in\mathbb{R}^m$. If 
\begin{equation}
\label{eq:primitive}
	\sum_i x_i\, \V(K_i[k],M[k-1],L,C_1,\ldots,C_{n-2k}) = 0
\end{equation}
holds for every $M\in\mathcal{P}(\Lambda)$, then
\begin{equation}
\label{eq:hr}
	(-1)^k 
	\sum_{i,j} x_i x_j\, \V(K_i[k],K_j[k],C_1,\ldots,C_{n-2k}) \ge 0.
\end{equation}
Moreover, the statement is nontrivial in the sense that for any 
$n\ge 2$ and $k\le n/2$, there is a simple polytope $\Lambda=
L=C_1=\cdots=C_{n-2k}$ in 
$\mathbb{R}^n$, $m\ge 1$, 
$K_1,\ldots,K_m\in\mathcal{P}(\Lambda)$, and 
$x\in\mathbb{R}^m$
so that \eqref{eq:primitive} holds and the inequality in \eqref{eq:hr} is 
strict.
\end{thm}

The case $k=1$ of Theorem \ref{thm:hodge} is nothing other than 
the Alexandrov-Fenchel inequality. To see why 
this is so, assume without loss of generality that $L=\sum_i y_i K_i$
for some $y\ge 0$, $y\ne 0$ (otherwise let
$m\leftarrow m+1$ and $K_{m+1}\leftarrow L$), and define 
$$
	\MM_{ij} = \V(K_i,K_j,C_1,\ldots,C_{n-2}).
$$
Then the statement of Theorem \ref{thm:hodge} for $k=1$ may be 
formulated as
$$
	\langle x,\MM y\rangle =0\quad\mbox{implies}\quad
	\langle x,\MM x\rangle \le 0
$$
for any $x$ and $y\ge 0$, $y\ne 0$. Thus by Lemma 
\ref{lem:hyp}, the inequality of Theorem \ref{thm:hodge} in the case $k=1$ is equivalent to 
the Alexandrov-Fenchel inequality for convex bodies in 
$\mathcal{P}(\Lambda)$. As any collection of convex bodies can be 
approximated by simple strongly isomorphic polytopes 
\cite[Theorem 2.4.15]{Sch14}, the general case of the Alexandrov-Fenchel 
inequality is further equivalent to this special case.

For $k>1$, the statement of Theorem \ref{thm:hodge} may be viewed as an 
analogue of the Alexandrov-Fenchel inequality for 
$\MM_{ij}=\V(K_i[k],K_j[k],C_1,\ldots,C_{n-2k})$. Thus the Hodge-Riemann 
relations are reminiscent of Conjecture \ref{conj:fed}, but their 
formulation is considerably more subtle. In section \ref{sec:main}, we 
will show that the Hodge-Riemann relations in fact contradict Conjecture 
\ref{conj:fed}, disproving the latter.

The aim of the rest of this section is to convince the reader that the 
statement of Theorem \ref{thm:hodge} given here in terms of mixed volumes 
is equivalent to the statement of the Hodge-Riemann relations as given in 
\cite{Tim99}. The reader who is primarily interested in Theorem 
\ref{thm:main} may safely jump ahead to section \ref{sec:main}.

To explain the formulation of \cite{Tim99}, we must first introduce 
some additional
notation. Let $u_1,\ldots,u_N\in S^{n-1}$ be the normal directions of the 
facets of $\Lambda$. For any $P\in\mathcal{P}(\Lambda)$, we denote by
$\hh_P\in\mathbb{R}^N$ its support vector
$$
	(\hh_P)_i := \sup_{y\in P} \langle y,u_i\rangle.
$$
Then there is a homogenous polynomial $V:\mathbb{R}^N\to\mathbb{R}$ of 
degree $n$, called the \emph{volume polynomial}, so that 
$\Vol(P)=V(\hh_P)$ 
for every $P\in\mathcal{P}(\Lambda)$ \cite[\S 5.2]{Sch14}. Moreover, as 
$\mathcal{P}(\Lambda)$ is closed under addition \cite[\S 2.4]{Sch14}, it 
follows immediately from the definition of mixed volumes that we have
for any $P_1,\ldots,P_n\in\mathcal{P}(\Lambda)$
$$
	\V(P_1,\ldots,P_n)=\frac{1}{n!}D_{\hh_{P_1}}
	\cdots D_{\hh_{P_n}}V,
$$
where $D_\hh$ denotes the directional derivative in direction $\hh$.
In this notation, the Hodge-Riemann relations are formulated in 
\cite[p.\ 385]{Tim99} as follows:

\begin{thm}
\label{thm:timorin}
Let $k\le n/2$, $L,C_1,\ldots,C_{n-2k}\in\mathcal{P}(\Lambda)$, and let
$\alpha=\sum_{|I|=k}\alpha_I D^I$ 
be a homogeneous differential operator of order $k$ with 
constant coefficients. If
\begin{equation}
\label{eq:timpr}
	\alpha D_{\hh_L}D_{\hh_{C_1}}\cdots
	D_{\hh_{C_{n-2k}}}V=0,
\end{equation}
then
\begin{equation}
\label{eq:timhr}
	(-1)^k\alpha^2 D_{\hh_{C_1}}\cdots
        D_{\hh_{C_{n-2k}}}V\ge 0.
\end{equation}
Moreover, equality is attained if and only if $\alpha V=0$.
\end{thm}

To write Theorem \ref{thm:timorin} in terms of mixed volumes, we 
need the following.

\begin{lem}
\label{lem:pol}
For any homogeneous differential operator 
$\alpha=\sum_{|I|=k}\alpha_I D^I$, there exist
$m\ge 1$, $K_1,\ldots,K_m\in\mathcal{P}(\Lambda)$, and
$x\in\mathbb{R}^m$ so that
$\alpha = \sum_i x_i (D_{\hh_{K_i}})^k$.
\end{lem}

\begin{proof}
We first recall that for any $z\in\mathbb{R}^N$, 
$\hh_\Lambda+\varepsilon z$ is the support vector of some polytope
$K\in\mathcal{P}(\Lambda)$ for sufficiently small $\varepsilon$ 
(as $\Lambda$ is simple, cf.\ \cite[Lemma 2.4.13]{Sch14}).
We may therefore write $z = \hh_{L}-\hh_{L'}$ where
$L=\varepsilon^{-1}K$ and $L'=\varepsilon^{-1}\Lambda$.

Now denote by $e_1,\ldots,e_N$ the standard coordinate basis of 
$\mathbb{R}^N$. By the above observation, we may write $e_i = 
\hh_{L_{i}}-\hh_{L_{i}'}$ for $L_{i},L_{i}'\in\mathcal{P}(\Lambda)$. We 
can therefore write
$$
	\alpha = \sum_{i_1\le\cdots\le i_k}
	\alpha_{i_1,\ldots,i_k}
	(D_{\hh_{L_{i_1}}}-D_{\hh_{L_{i_1}'}})\cdots
	(D_{\hh_{L_{i_k}}}-D_{\hh_{L_{i_k}'}}).
$$
By expanding the product, we may evidently express $\alpha$ as a linear 
combination of differential operators of the form
$D_{\hh_{R_{i_1}}}\cdots D_{\hh_{R_{i_k}}}$ with
$R_i\in\mathcal{P}(\Lambda)$. But as
$$
	D_{\hh_{R_{i_1}}}\cdots D_{\hh_{R_{i_k}}} =
	\frac{1}{k!}\sum_{\delta\in\{0,1\}^k}
	(-1)^{k+\delta_1+\cdots+\delta_k}
	(D_{\hh_{\delta_1R_{i_1}+\cdots+\delta_k R_{i_k}}})^k
$$
by the polarization formula \cite[p.\ 137]{BZ88}, the proof 
is readily concluded.
\end{proof}

We are now ready to show that the Hodge-Riemann relations expressed by 
Theorems \ref{thm:hodge} and \ref{thm:timorin} are equivalent.
First, note that $\alpha D_{\hh_L}D_{\hh_{C_1}}\cdots D_{\hh_{C_{n-2k}}}V$ 
in \eqref{eq:timpr} is a homogeneous polynomial of degree $k-1$. Thus 
\eqref{eq:timpr} is equivalent to the statement that 
$\beta\alpha D_{\hh_L}D_{\hh_{C_1}}\cdots D_{\hh_{C_{n-2k}}}V=0$ for every 
homogeneous differential operator $\beta$ of order $k-1$. By Lemma 
\ref{lem:pol}, the statement of Theorem 
\ref{thm:timorin} (without the equality case) may be equivalently 
formulated as follows: if
$$
	\alpha (D_{\hh_M})^{k-1} D_{\hh_L}
	D_{\hh_{C_1}}\cdots D_{\hh_{C_{n-2k}}}V=0
$$
for all $M\in\mathcal{P}(\Lambda)$, then 
$$
        (-1)^k\alpha^2 D_{\hh_{C_1}}\cdots
        D_{\hh_{C_{n-2k}}}V\ge 0.	
$$
That \eqref{eq:timpr}--\eqref{eq:timhr} imply
\eqref{eq:primitive}--\eqref{eq:hr} follows immediately by choosing
the differential operator
$\alpha = \sum_i x_i (D_{\hh_{K_i}})^k$.
Conversely, that \eqref{eq:primitive}--\eqref{eq:hr} imply  
\eqref{eq:timpr}--\eqref{eq:timhr} follows as any $\alpha$ can be 
expressed as $\alpha = \sum_i x_i (D_{\hh_{K_i}})^k$ by Lemma 
\ref{lem:pol}.

It remains to check that the Hodge-Riemann relations are nontrivial. This 
is certainly not obvious at first sight: the condition 
\eqref{eq:primitive} is a very strong one (as it must hold for \emph{any} 
$M\in\mathcal{P}(\Lambda)$), and it is not clear \emph{a priori} that it 
can be satisfied in any nontrivial situation. To show this is the case, 
consider the special case where $L=C_1=\cdots=C_{n-2k}=\Lambda$, and 
define the spaces
$$
	P_k := \{\alpha:\alpha(D_{\hh_\Lambda})^{n-2k+1}V=0\},\qquad\quad
	I := \{\alpha:\alpha V=0\}.
$$
The remarkable combinatorial theory underlying the Hodge-Riemann relations 
enables us to compute \cite[Corollary 5.3.4]{Tim99}
$$
	\dim(P_k/I) = h_k - h_{k-1},
$$
where $(h_1,\ldots,h_n)$ is the so-called $h$-vector of $\Lambda$. To show 
the Hodge-Riemann relations are nontrivial, it suffices to construct a 
simple polytope $\Lambda$ in $\mathbb{R}^n$ whose $h$-vector satisfies
$h_k>h_{k-1}$ for $k\le n/2$, as by Theorem \ref{thm:timorin} this ensures 
the existence of $\alpha$ so that \eqref{eq:timpr} holds and the 
inequality in \eqref{eq:timhr} is strict (by Lemma \ref{lem:pol}, this 
implies the corresponding statement of Theorem \ref{thm:hodge} for some 
$m,K_1,\ldots,K_m,x$).
But such an example is easily identified: e.g., we may choose 
$\Lambda$ to be the unit cube in $\mathbb{R}^n$, whose $h$-vector
is given by $h_k={n\choose k}$ by the computations in \cite[p.\ 
387]{Tim99} (note that $\Lambda=[0,1]\times\cdots\times[0,1]$ and
use the product formula for $H$-polynomials).

\section{Proof of Theorem \ref{thm:main}}
\label{sec:main}

We first consider the special case that $k=2$.

\begin{proof}[Proof of Theorem \ref{thm:main} for $k=2$]
Fix any $n\ge 4$ and let $k=2$. By the second part of Theorem 
\ref{thm:hodge}, we may choose a simple polytope 
$\Lambda=L=C_1=\cdots=C_{n-4}$ in 
$\mathbb{R}^n$, $m\ge 1$, polytopes
$K_1,\ldots,K_m\in\mathcal{P}(\Lambda)$, 
and $x\in\mathbb{R}^m$ so that \eqref{eq:primitive} holds and the 
inequality in \eqref{eq:hr} is strict. In the following, we will denote
$K_{m+1}:=\Lambda$.

Now define the $(m+1)\times (m+1)$ matrix
$$
        \MM_{ij} := \V(
        K_i[2],K_j[2],\Lambda[n-4]),
$$
and let $y=e_{m+1}$. Then \eqref{eq:primitive} with $M=\Lambda$ implies 
$$
	\langle x,\MM y\rangle=0,
$$
while the strict inequality in \eqref{eq:hr} may be written as
$$
	\langle x,\MM x\rangle>0.
$$
Note that $\MM$ is a positive matrix, as all bodies in 
$\mathcal{P}(\Lambda)$ are full-dimensional. Thus $\MM$ is not hyperbolic 
by Lemma \ref{lem:hyp}. In particular, by Lemma \ref{lem:syl}, there 
exists $I\subseteq[m+1]$ so that $(-1)^{|I|}\det \MM_I>0$. The latter 
contradicts Conjecture \ref{conj:fed}.
\end{proof}

Informally, the above proof works as follows. By Lemma \ref{lem:syl}, 
Fedotov's Conjecture \ref{conj:fed} is equivalent to the statement that 
the matrix $\MM$ is hyperbolic. However, when $k=2$, the Hodge-Riemann 
relation \eqref{eq:hr} yields an inequality in the opposite direction from 
the one that holds for hyperbolic matrices by Lemma \ref{lem:hyp}. Thus 
the Hodge-Riemann relation contradicts Fedotov's conjecture.

Precisely the same argument works whenever $k\ge 2$ is even. Curiously, 
however, the argument fails when $k$ is odd, as then \eqref{eq:hr} and 
hyperbolicity yield inequalities in the same direction. To prove Theorem 
\ref{thm:main} for arbitrary $k$, we will use a different argument: rather 
than applying the Hodge-Riemann relation of degree $k$, we will instead 
reduce the problem for any $k>2$ back to the case $k=2$.

\begin{proof}[Proof of Theorem \ref{thm:main} for general $k$]
Fix any $n\ge 6$ and $2<k\le n/2$. Choose $\Lambda$, $m$, 
$K_1,\ldots,K_{m+1}$, $x,y$, and $\MM$ as in the proof of the $k=2$ case. 
Note first that
\begin{align*}
	\MM_{ij} &:=
	\V(K_i[2],K_j[2],\Lambda[n-4]) \\ &\phantom{:}=
	\V(K_i[2],\Lambda[k-2],K_j[2],\Lambda[k-2],
	\Lambda[n-2k])
	\\
	&\phantom{:}=
	\frac{1}{(k!)^2}\sum_{\delta,\varepsilon\in\{0,1\}^k}
	(-1)^{k+\delta_1+\cdots+\delta_k}
	(-1)^{k+\varepsilon_1+\cdots+\varepsilon_k}
	\V(K_{i\delta}[k],
	K_{j\varepsilon}[k],\Lambda[n-2k])
\end{align*}
by the polarization formula \cite[p.\ 137]{BZ88}, where
$$
	K_{i\delta} 
	:=(\delta_1+\delta_2)K_i+(\delta_3+\cdots+\delta_k)\Lambda.
$$
Define the $(m+1)(2^k-1)\times (m+1)(2^k-1)$ positive matrix
$$
	\tilde\MM_{i\delta,j\varepsilon} :=
	\V(K_{i\delta}[k],
        K_{j\varepsilon}[k],\Lambda[n-2k])
$$
for $i,j\in[m+1]$, $\delta,\varepsilon\in 
\{0,1\}^k\backslash(0,\ldots,0)$,
and define $\tilde x,\tilde y\in\mathbb{R}^{(m+1)(2^k-1)}$ as
$$
	\tilde x_{i\delta} = 
	\frac{(-1)^{k+\delta_1+\cdots+\delta_k} x_i}{k!},
	\qquad\quad
	\tilde y_{i\delta} = 1_{i=m+1} 1_{\delta=(1,0,\ldots,0)}.
$$
Then
$$
	\langle \tilde x,\tilde\MM\tilde y\rangle =
	\langle x,\MM y\rangle = 0,
	\qquad\quad
	\langle \tilde x,\tilde\MM\tilde x\rangle = 
	\langle x,\MM x\rangle > 0,
$$
so $\tilde\MM$ cannot be hyperbolic. The latter contradicts Conjecture 
\ref{conj:fed} for the given value of $k$ as in the proof of the case 
$k=2$.
\end{proof}

\section{An explicit example}
\label{sec:explicit}

The proof of Theorem \ref{thm:main} shows that counterexamples to 
Fedotov's conjecture are prevalent: any simple polytope $\Lambda$ whose 
Hodge-Riemann relation of degree $2$ is nontrivial (that is, whose 
$h$-vector satisfies $h_2>h_1$, cf.\ section \ref{sec:hodge}) gives rise 
to a counterexample to Conjecture \ref{conj:fed} with 
$C_1=\cdots=C_{n-2k}=\Lambda$ and some $K_1,\ldots,K_m$ strongly 
isomorphic to $\Lambda$. However, the construction itself is rather 
indirect. The aim of this section is to illustrate the construction by 
means of a simple explicit example in the case that $\Lambda$ is the unit 
cube.

Let $\Lambda=[0,e_1]+\cdots+[0,e_n]$ be the unit cube in 
$\mathbb{R}^n$. Then any $M\in\mathcal{P}(\Lambda)$ is a parallelepiped 
of the form $M=M_a+v$ for some $a_1,\ldots,a_n>0$ and $v\in\mathbb{R}^n$, 
where
$$
	M_a:=a_1[0,e_1]+\cdots+a_n[0,e_n].
$$
By translation-invariance of mixed volumes, it suffices to consider $v=0$. 
We can compute mixed volumes of parallelepipeds using that
$$
	n!\,\V([0,e_{i_1}],\ldots,[0,e_{i_n}]) 
	= 1_{i_1\ne \cdots\ne i_n}
$$
by \cite[(5.77)]{Sch14}, so that by additivity of mixed volumes
$$
	n!\,\V(M_{a^{(1)}},\ldots,M_{a^{(n)}}) =
	\sum_{i_1\ne \cdots\ne i_n}
	a^{(1)}_{i_1}\cdots a^{(n)}_{i_n}.
$$
Using this simple expression, it is not difficult to generate explicit 
examples.

For example, for the case $n=4$, $k=2$, let us define
\begin{align*}
	K_1 &:= [0,e_1]+[0,e_2], \\
	K_2 &:= [0,e_3]+[0,e_4], \\
	K_3 &:= [0,e_1]+\cdots+[0,e_4]=\Lambda.
\end{align*}
Then it is readily verified by means of the above formula that
$$
	3\,\V(K_1[2],M,\Lambda) +
	3\,\V(K_2[2],M,\Lambda) - \V(K_3[2],M,\Lambda) = 0
$$
for all $M\in\mathcal{P}(\Lambda)$, that is, \eqref{eq:primitive} holds 
with $x_1=x_2=3$ and $x_3=-1$. (This is most easily seen by using 
$\Lambda=K_1+K_2$ and $\V(K_1[3],M)=\V(K_2[3],M)=0$ for all 
$M$.) On the other hand, we compute
\begin{multline*}
	\sum_{i,j} x_ix_j\,\V(K_i[2],K_j[2]) =
	18\, \V(K_1[2],K_2[2]) - 6\, \V(K_1[2],K_3[2])  \\
	-6\, \V(K_2[2],K_3[2]) + \V(K_3[2],K_3[2]) = 2,
\end{multline*}
so that \eqref{eq:hr} holds with strict inequality. 
It therefore follows from the argument in the proof of Theorem 
\ref{thm:main} that Conjecture 
\ref{conj:fed} must fail for $n=4$, 
$k=2$, $m=3$ when $K_1,K_2,K_3$ are chosen as above.
The author is indebted to the anonymous referee of this note for 
suggesting this example.

\begin{rem}
Technically speaking the above example does not verify the assumptions of 
Theorem \ref{thm:hodge}, as $K_1,K_2$ have empty interior and are 
therefore not strongly isomorphic to $\Lambda$. However, the example 
remains valid if we replace $K_1,K_2,x_3$ by $K_1'=K_1+\varepsilon K_2$,
$K_2'=K_2+\varepsilon K_1$, and $x_3'=-1-4\varepsilon-\varepsilon^2$
for any $\varepsilon>0$.
\end{rem}

Of course, given any explicit example, one can readily verify directly 
that Conjecture \ref{conj:fed} fails without any reference to the 
Hodge-Riemann relations. However, this obscures the fundamental reason for 
the failure of Fedotov's conjecture which was essential for the discovery 
of such counterexamples. On the other hand, the above explicit example 
provides additional information beyond our main result as stated in 
Theorem \ref{thm:main}: it shows that Fedotov's conjecture fails already 
when $k=2$ and $m=3$, that is, in the smallest case that is not covered by 
Lemma \ref{lem:fedeasy}. The example is readily modified to extend this 
conclusion to any $k$.

\begin{lem}
For every $k\ge 2$ and $n\ge 2k$, Conjecture \ref{conj:fed} fails for
$m=3$.
\end{lem}

\begin{proof}
Define the following bodies:
\begin{align*}
	&K_1 = [0,e_1]+\cdots+[0,e_k], \\
	&K_2 = [0,e_{k+1}]+\cdots+[0,e_{2k}],\\
	&K_3 = [0,e_1]+\cdots+[0,e_{2k}],\\
	&C_1,\ldots,C_{n-2k} = [0,e_{2k+1}]+\cdots+[0,e_n].
\end{align*}
Then we can compute 
$\MM_{ij}:=\V(K_i[k],K_j[k],C_1,\ldots,C_{n-2k})$ explicitly as
$$
	\MM = \begin{bmatrix}
	0 & a & a \\
	a & 0 & a \\
	a & a & b
	\end{bmatrix},
	\qquad
	a=\frac{(k!)^2(n-2k)!}{n!},\qquad b=\frac{(2k)!(n-2k)!}{n!}.
$$
Therefore
$$
	\det\MM = a^2(2a - b) =
	\frac{(k!)^4((n-2k)!)^3}{(n!)^3}
	(2(k!)^2-(2k)!)<0
$$
whenever $k\ge 2$, contradicting Conjecture \ref{conj:fed}.
\end{proof}

\begin{rem}
The explicit expression for $n!\,\V(M_{a^{(1)}},\ldots,M_{a^{(n)}})$ given 
above is nothing other than the permanent of the matrix whose columns are 
$a^{(1)},\ldots,a^{(n)}$. It is well known \cite[\S 25.4]{BZ88} that the 
permanent of a matrix is not only a special case of mixed volumes, but 
also of mixed discriminants (the linear-algebraic analogue of mixed 
volumes). The above example therefore shows that the 
analogue of Fedotov's conjecture for mixed discriminants is also invalid. 
This should not come as a surprise, as mixed discriminants also satisfy 
Hodge-Riemann relations \cite{Tim98} and thus the arguments behind Theorem 
\ref{thm:main} extend to this situation. 
\end{rem}

\section{Hodge-Riemann relations fail for general convex bodies}
\label{sec:hrgeneral}

Beside the disproof of Fedotov's conjecture, an expository aim of this 
note has been to highlight that the Hodge-Riemann relations of McMullen 
and Timorin may be interpreted entirely in terms of familiar objects from 
classical convex geometry: they provide inequalities between mixed volumes 
that generalize the Alexandrov-Fenchel inequality. From the viewpoint of 
classical convexity, however, the formulation of Theorem \ref{thm:hodge} 
exhibits a puzzling aspect. In principle, the statements of the relations
\eqref{eq:primitive} and \eqref{eq:hr} make sense when $K_i,C_i,M,L$ are 
arbitrary convex bodies, but the statement of Theorem \ref{thm:hodge} 
requires these bodies to be strongly isomorphic simple polytopes. It is 
not immediately clear why the latter is important: most classical inequalities 
in convex geometry are either valid for arbitrary convex bodies, or 
involve geometric quantities that do not make sense in the absence of
regularity conditions (such as uniform bounds on the principal curvatures).

We have shown in section \ref{sec:hodge} that the Hodge-Riemann relation 
of degree $k=1$ is equivalent to the Alexandrov-Fenchel inequality for 
strongly isomorphic polytopes. The inequality then extends readily to 
arbitrary convex bodies by approximation. This is possible because for 
$k=1$ the relations \eqref{eq:primitive} and \eqref{eq:hr} can be combined 
into a single inequality by Lemma \ref{lem:hyp}, and this 
\emph{inequality} is preserved by taking limits. However, a natural 
analogue of Lemma \ref{lem:hyp} does not hold for $k\ge 2$. It is 
therefore unclear how to apply an approximation argument, as the 
\emph{equality} \eqref{eq:primitive} need not be stable under 
approximation (that is, if \eqref{eq:primitive} holds for a given 
collection of convex bodies, they might not be approximated by simple 
strongly isomorphic polytopes in such a way that \eqref{eq:primitive} 
remains valid for the approximations).

We will presently show by means of a simple example that the Hodge-Riemann 
relation of degree $k=2$ can in fact fail for general convex bodies.

\begin{example} 
Let $B$ be the 
Euclidean unit ball in $\mathbb{R}^4$, and let $L=\mathrm{conv}\{B,x\}$ 
for some $x\not\in B$, that is, $L$ is a cap body of $B$. It is a 
classical fact, which dates back essentially to Minkowski, that 
\cite[Theorem 7.6.17]{Sch14}
$$
	\V(L,L,B,L) =
	\V(B,L,B,L) = 
	\V(B,B,B,L) >
	\V(B,B,B,B).
$$
In particular, this gives rise to a nontrivial equality case of the 
Alexandrov-Fenchel inequality of Theorem \ref{thm:af} with $n=4$, 
$K=C_1=B$, $C_2=L$. The latter implies
$$
	\V(M,B,B,L) = \V(M,L,B,L)
$$
for all convex bodies $M$, cf.\ 
\cite[Theorem 7.4.3]{Sch14} or \cite[Lemma 3.12]{SvH20}.

We will now use these observations to construct a counterexample to the 
Hodge-Riemann relation of degree $k=2$ for general convex bodies. Define
$$
	K_1:=B,\qquad\quad K_2:=L,\qquad\quad K_3:=B+L.
$$
Then
\begin{multline*}
	3\,\V(K_1[2],M,L) + 
	\V(K_2[2],M,L) -
	\V(K_3[2],M,L) = \\
	2\,\V(M,B,B,L) - 2\,\V(M,L,B,L) = 0
\end{multline*}
for all convex bodies $M$; that is, \eqref{eq:primitive} is satisfied
with $x_1=3$, $x_2=1$, and $x_3=-1$.
On the other hand, we can compute
$$
	\sum_{i,j} x_ix_j\,\V(K_i[2],K_j[2]) =
	4\, \V(B,B,B,B) - 4\, \V(L,B,B,B) < 0,
$$
contradicting the validity of \eqref{eq:hr}.
\end{example}

\begin{rem}
There is nothing special about the particular choice of the Euclidean
ball in this example: the conclusion remains valid when $B$ is replaced
by an arbitrary convex body $K$ and $L$ is a cap body of $K$ as defined in 
\cite[p.\ 87]{Sch14}. For example, we may take $L$ to be the unit cube in 
$\mathbb{R}^4$ and $K$ to be the same cube with one of its corners sliced 
off. The latter variant of the example shows that the Hodge-Riemann 
relations can fail for polytopes that are not strongly isomorphic.
\end{rem}

The above example suggests that the validity of Hodge-Riemann relations of 
degree $k\ge 2$ is related to the study of the equality cases of the 
Alexandrov-Fenchel inequality: indeed, the assumption \eqref{eq:primitive} 
is reminiscent of the equality condition of the Alexandrov-Fenchel 
inequality (cf. \cite[Theorem 7.4.2]{Sch14}), which is precisely what was 
used to construct the above counterexample. Even though the 
Alexandrov-Fenchel inequality is stable under approximation, this cannot 
be used to study its nontrivial equality cases as the latter are destroyed 
by approximation \cite{Sch85,SvH19,SvH20}. The above example shows that 
for Hodge-Riemann relations of degree $k\ge 2$, this instability is 
manifested even by the inequality itself.

On the other hand, it is expected that the validity of Hodge-Riemann 
relations should extend to ``ample'' families of convex bodies other than 
simple strongly isomorphic polytopes. In particular, one may conjecture
that the statement of Theorem \ref{thm:hodge} remains valid if the class 
$\mathcal{P}(\Lambda)$ is replaced by the class $C^\infty_+$ of convex 
bodies whose boundaries are smooth and have strictly positive curvature. 
Some initial progress in this direction may be found in the recent papers
\cite{Kot20,Ale20,KW22}.

\subsection*{Acknowledgments}

This work was supported in part by NSF grants DMS-1811735 and DMS-2054565, 
and by the Simons Collaboration on Algorithms \& Geometry. The author is 
grateful to Jan Kotrbat\'y for bringing Fedotov's conjecture to his 
attention, and to the anonymous referee for very helpful comments on the 
first version of this note and for suggesting the explicit example of 
section \ref{sec:explicit}.

\bibliographystyle{abbrv}
\bibliography{ref}

\end{document}